\documentclass[reqno]{amsart}
\usepackage{amsthm,amsmath,amssymb}
\usepackage{stmaryrd,mathrsfs}
\usepackage{enumerate}
\usepackage[T1]{fontenc}
\usepackage{esint}
\usepackage{tikz}
\usepackage{mathtools}

\usepackage{hyperref}

\usepackage{color}

\allowdisplaybreaks

\newcommand{\ud}[0]{\,\mathrm{d}}

\newcommand{\dist}[0]{\operatorname{dist}}
\newcommand{\PV}{\operatorname{p.v.}}

\newcommand{\abs}[1]{|#1|}
\newcommand{\Babs}[1]{\Big|#1\Big|}

\newcommand{\Norm}[2]{\|#1\|_{#2}}

\newcommand{\ave}[1]{\langle #1\rangle}
\newcommand{\Dini}{\operatorname{Dini}}

\newcommand{\supp}[0]{\operatorname{supp}}
\newcommand{\loc}[0]{\operatorname{loc}}

\newcommand{\R}{\mathbb{R}}
\newcommand{\C}{\mathbb{C}}
\newcommand{\N}{\mathbb{N}}
\newcommand{\Z}{\mathbb{Z}}

\newcommand{\Sp}{\mathbb{S}}
\newcommand{\Qs}{\mathscr{Q}}
\newcommand{\Ss}{\mathscr{S}}

\newcommand{\Bc}{\mathcal{B}}
\newcommand{\Pc}{\mathcal{P}}
\newcommand{\Rc}{\mathcal{R}}
\newcommand{\Tc}{\mathcal{T}}

\newcommand{\D}[0]{\mathscr{D}}
\newcommand{\eps}[0]{\varepsilon}

\newcommand{\A}{\mathcal{A}}
\newcommand{\ch}{\text{ch}}

\theoremstyle{plain}
\newtheorem{theorem}[equation]{Theorem}

\newtheorem{corollary}[equation]{Corollary}
\newtheorem{lemma}[equation]{Lemma}
\newtheorem{conjecture}[equation]{Conjecture}

\theoremstyle{definition}

\newtheorem{remark}[equation]{Remark}

\numberwithin{equation}{section}

\begin{document}

\title[Quantitative weighted estimates]{Quantitative weighted estimates for rough homogeneous singular integrals}

\author[T. P. Hyt\"onen, L. Roncal, and O. Tapiola]{Tuomas P.\ Hyt\"onen, Luz Roncal, and Olli Tapiola}

\address[T. P. Hyt\"onen and O. Tapiola]{Department of Mathematics and Statistics, P.O.B.~68 (Gustaf H\"all\-str\"omin katu~2b), FI-00014 University of Helsinki, Finland}

\email{\{tuomas.hytonen,olli.tapiola\}@helsinki.fi}

\address[L. Roncal]{Departamento de Matem\'aticas y Computaci\'on, Universidad de La Rioja, C/Luis de Ulloa s/n, 26004 Logro\~no, Spain}

\email{luz.roncal@unirioja.es}

\thanks{T.H. and O.T. are supported by the ERC Starting Grant ANPROB. They are members of the Finnish Centre of Excellence in Analysis and Dynamics Research. L.R. is partially supported by grant MTM2012-36732-C03-02 from Spanish Government and the mobility grant ``Jos\'e Castillejo'' number CAS14/00037 from Ministerio de Educaci\'on, Cultura y Deporte of Spain.}


\maketitle

\begin{abstract}
We consider homogeneous singular kernels, whose angular part is bounded, but need not have any continuity. For the norm of the corresponding singular integral operators on the weighted space $L^2(w)$, we obtain a bound that is quadratic in the $A_2$ constant $[w]_{A_2}$. We do not know if this is sharp, but it is the best known quantitative result for this class of operators. The proof relies on a classical decomposition of these operators into smooth pieces, for which we use a quantitative elaboration of Lacey's dyadic decomposition of Dini-continuous operators: the dependence of constants on the Dini norm of the kernels is crucial to control the summability of the series expansion of the rough operator. We conclude with applications and conjectures related to weighted bounds for powers of the Beurling transform.
\end{abstract}

\section{Introduction and main results}

We are concerned with sharp weighted inequalities for singular integral operators, a topic that goes back to \cite{AIS,Pet-Vol} in the case of the Beurling operator, continues through the solution of the $A_2$ conjecture for all standard Calder\'on--Zygmund operators  \cite{Hytonen:A2} and the alternative approach to this result by A. K. Lerner \cite{Ler1,Ler2}, and keeps developing with new extensions, among them the recent approach of M. T. Lacey \cite{Lacey} covering all Dini-continuous kernels. For a more precise discussion of the background and our contributions, we need to recall some definitions:

Let $T$ be a bounded linear operator on $L^2(\R^d)$ represented as
$$
Tf(x)=\int_{\R^d}K(x,y)f(y)\,dy, \quad \forall x\notin \supp f.
$$
A function $\omega \colon [0, \infty) \to [0,\infty)$ is a \textit{modulus of continuity} if it is increasing and subadditive
(i.e. $\omega(t + s) \le \omega(t) + \omega(s)$) and $\omega(0)=0$. We say that the operator $T$ above is an \emph{$\omega$-Calder\'on--Zygmund operator}
if the kernel $K$ has the standard size estimate
\begin{equation}
\label{eq:CK}
|K(x,y)|\le \frac{C_K}{|x-y|^d}
\end{equation}
and the smoothness estimate
$$
|K(x,y)-K(x',y)|+|K(y,x)-K(y,x')|\le \omega\bigg(\frac{|x-x'|}{|x-y|}\bigg)\frac{1}{|x-y|^d}
$$
for $|x-y|>2|x-x'|>0$.
(We deliberately leave out any multiplicative constant from the smoothness estimate, as this can be incorporated into the function $\omega$.)
Moreover, $K$ is said to be a \textit{Dini-continuous kernel} if $\omega$ satisfies the \textit{Dini condition}:
\begin{equation}
\label{eq:Dini}
\|\omega\|_{\Dini}:=\int_0^1\omega(t)\frac{dt}{t}<\infty.
\end{equation}

Let us denote the average of a function $f$ over a cube $Q$ by
$$
  \ave{f}_Q = \fint_Q f = \fint_Q f(x)\,dx =\frac{1}{|Q|}\int_Qf(x)\,dx.
$$
Here $|Q|$ is the Lebesgue measure of $Q$.
A weight is a nonnegative and finite almost everywhere function on $\R^d$. For $1<p<\infty$, the Muckenhoupt class $A_p$ is the set of locally integrable weights $w$ for which $w^{1-p'}\in L^1_{\loc}(\R^d)$, with $1/p+1/p'=1$, and
\begin{equation*}
[w]_{A_p}:=\sup_Q\Big(\fint_Q w \Big)\Big(\fint_Q w^{1-p'}\Big)^{p-1}<\infty,
\end{equation*}
where the supremum is taken over all cubes in $\R^d$.
We will adopt the following definition for the $A_{\infty}$ constant of a weight $w$ introduced by N. Fujii \cite{Fujii}, and later by J.M. Wilson \cite{Wilson}:
$$
[w]_{A_{\infty}}:= \sup_Q\frac{1}{w(Q)}\int_{Q}M(1_Q w)(x)\,dx.
$$
Here, $w(Q):=\int_Qw(x)\,dx$, $1_Qw(x)=w(x)1_Q(x)$, where $1_Q$ is the characteristic function of $Q$, and the supremum above is taken over all cubes with edges parallel to the coordinate axes. When the supremum is finite, we will say that $w$ belongs to the $A_{\infty}$ class.

Our weighted estimates are most efficiently stated in terms of the following variants of the weight characteristic:
\begin{align*}
    \{w\}_{A_p} &:=[w]_{A_p}^{1/p}\max\{[w]_{A_\infty}^{1/p'},[w^{1-p'}]_{A_{\infty}}^{1/p}\}, \\
    (w)_{A_p} &:=\max\{[w]_{A_\infty},[w^{1-p'}]_{A_\infty}\}.
\end{align*}
Using the facts that $[w^{1-p'}]_{A_{p'}}^{1/p'}=[w]_{A_p}^{1/p}$ and $[w]_{A_\infty}\leq c_d[w]_{A_p}$ (for the latter, see \cite[last display on p.~778]{HP}), one easily checks that
\begin{equation*}
  (w)_{A_p}\leq c_d\{w\}_{A_p}\leq c_d'[w]_{A_p}^{\max\{1,1/(p-1)\}},
\end{equation*}
so any bounds in terms of the weight characteristics $(w)_{A_p}$ and $\{w\}_{A_p}$ are stronger than similar bounds using $[w]_{A_p}^{\max\{1,1/(p-1)\}}$, which is the expression that most frequently appears in sharp estimates for Calder\'on--Zygmund operators, like \cite{Hytonen:A2}. While our notation is non-standard, we note that the characteristics $(w)_{A_p}$ and $\{w\}_{A_p}$ have implicitly appeared in several recent contributions, starting from \cite{HP,Lacey:Houston}.

Given a Calder\'on--Zygmund operator $T$, the \emph{maximal truncation of $T$} is the operator $T_{\sharp}$ given by
\begin{equation*}
  T_{\sharp}f(x):=\sup_{\eps>0}\Big|T_\eps f(x)\Big|.
\end{equation*}
where $T_\eps$ is the \emph{$\eps$-truncation of $T$}:
$$
T_{\eps}f(x):=\int_{|x-y|>\eps}K(x,y)f(y)\,dy.
$$

Our first main result in contained in the following. It is a fully quantitative version of a recent theorem of Lacey~\cite{Lacey}, which in turn is an extension of the $A_2$ theorem of the first author~\cite{Hytonen:A2}.

\begin{theorem}
\label{thm:main1}
Let $T$ be an $\omega$-Calder\'on--Zygmund operator whose modulus of continuity satisfies the Dini condition \eqref{eq:Dini}. Let $1<p<\infty$. Then, for every $w\in A_p$, we have
$$
\|T_{\sharp}\|_{L^p(w)\to L^p(w)}\le c_{d,p}\big(\|T\|_{L^2\to L^2}+C_K+\|\omega\|_{\Dini}\big) \{w\}_{A_p}.
$$
In particular,
$$
\|T_{\sharp}\|_{L^2(w)\to L^2(w)}\le c_d\big(\|T\|_{L^2\to L^2}+C_K+\|\omega\|_{\Dini}\big) [w]_{A_2}.
$$
\end{theorem}

Lacey proves such a result under the same assumptions on the operator~$T$, but without specifying the dependence of the norm bound on the Calder\'on--Zygmund characteristics of $T$. For us, the precise form of this dependence will be important in the application to our second main result. It should be noted that obtaining the stated dependence, especially on $\|\omega\|_{\Dini}$, is not just a question of keeping careful track of constants in Lacey's proof, but requires some new twist in the argument: Lacey's proof relies on the possibility of making $\int_0^\rho\omega(t)\,dt/t$ small by choosing $\rho$ small enough, and in this way it introduces a more complicated implicit dependence on the function $\omega$.

We now recall the notion of \textit{rough homogeneous singular integrals}. These are operators with convolution kernels $K(x,y)=K(x-y)$ where, writing $x'=x/|x|$,
$$
K(x)=\frac{\Omega(x')}{|x|^d}, \qquad \Omega\in L^{\infty}(\Sp^{d-1}),
$$
and
\begin{equation*}
\int_{\Sp^{d-1}}\Omega\,d\sigma=0.
\end{equation*}
Observe that the kernel so defined is homogeneous of degree $-d$.
The size estimate is as usual, but there is no angular smoothness. Then, we will write
$$
T_{\Omega}f(x)=\PV\int_{\R^d}\frac{\Omega(y')}{|y|^d}f(x-y)\,dy=
\lim_{\substack{\eps\to 0\\R\to\infty}}\int_{\eps<|y|<R}\frac{\Omega(y')}{|y|^d}f(x-y)\,dy.
$$
Our second aim is to prove the following.

\begin{theorem}
\label{thm:main2}
For every $w\in A_p$, we have
\begin{equation*}
\|T_{\Omega}\|_{L^p(w)\to L^p(w)}
\le  c_{d,p}\|\Omega\|_{L^\infty}  \{w\}_{A_p} (w)_{A_p}.
\end{equation*}
In particular,
\begin{equation*}
\|T_{\Omega}\|_{L^2(w)\to L^2(w)}
\le  c_d\|\Omega\|_{L^\infty}  [w]_{A_2}^2.
\end{equation*}
\end{theorem}

Qualitatively, without specifying the dependence of the norm bound on the $A_2$ characteristics of $w$, this result is well known \cite{DuoTAMS, Duo-RdF, Watson}. The question of sharp dependence on $[w]_{A_2}$ for $T_\Omega$ was raised during the workshop ``Weighted singular integral operators and non-homogenous harmonic analysis'' at the American Institute of Mathematics (Palo Alto, California) in October 2011, and the first author discussed this issue especially with David Cruz-Uribe. Tracking the dependence in one of the classical proofs of a qualitative form of Theorem~\ref{thm:main2}, we arrived at a somewhat higher power on $[w]_{A_2}$ back then. We do not know if the bound above is optimal but, to our knowledge, it is the best that is currently available.

\subsection*{Notation} By $c_d$ we mean a positive dimensional constant. Also, the positive constants not depending on the essential variables will be denoted by $C$. Both $C$ and $c_d$ may vary at each occurrence. For $x\in \R^d$, $r>0$, the ball of center $x$ and radius $r$ is the set $B(x,r):=\{y\in \R^d: |x-y|<r\}$. For an operator $T$, $\|T\|_{B_1\to B_2}$ is the operator norm, that is, the smallest $N$ in the inequality $\|Tf\|_{B_2}\le N\|f\|_{B_1}$. Sometimes we will use the notation $a \vee b \coloneqq \max\{a,b\}$.
Finally, given a function $f$, by $\widehat{f}$ we will denote the Fourier transform of $f$.

\section{Calder\'on--Zygmund operators with Dini-continuous kernel}

Recently, Lacey \cite[Theorem 4.2]{Lacey} extended the $A_2$ theorem to a more general class of Calder\'on--Zygmund operators, whose modulus of continuity $\omega$ satisfies the Dini condition \eqref{eq:Dini}. (Very recently, his method has been pushed even further in \cite{BFP}, but this extension goes to a different direction than our present needs.)
For such operators, Lacey proved a pointwise domination theorem by so-called sparse operators, which originate from the approach to the $A_2$ theorem due to Lerner \cite{Ler1,Ler2}.

However, Lacey's result was qualitative in the sense that the constants arising were not fully explicit in terms of $\omega$. In this section, we revisit Lacey's results, and show the precise
quantitative dependence on the Dini condition in the pointwise domination result. As a consequence, we will obtain Theorem \ref{thm:main1} as a corollary.

\subsection{Dyadic cubes, adjacent dyadic systems and sparse operators}

We begin with some necessary definitions. The \emph{standard system of dyadic cubes} in $\R^d$ is the collection $\D$,
$$
\mathscr{D}:=\big\{2^{-k}([0,1)^d+m):k\in \Z, m\in \Z^d\big\},
$$
consisting of simple half-open cubes of different length scales with sides parallel to the coordinate axes. These cubes satisfy the following three important properties:
\begin{enumerate}
\item[1)] for any $Q\in\mathscr{D}$, the sidelength $\ell (Q)$ is of the form $2^k$, $k\in \Z$;
\item[2)] $Q\cap R\in\{Q,R,\emptyset\}$, for any $Q,R\in\mathscr{D}$;
\item[3)] the cubes of a fixed sidelength $2^k$ form a partition of $\R^d$.
\end{enumerate}

Although the standard system of dyadic cubes is a versatile tool in mathematical analysis, it does have some disadvantages. Namely, if $B(x,r)$ is a ball, then
there does not usually exist a cube $Q \in \mathscr{D}$ such that $B(x,r) \subset Q$ and $\ell(Q) \approx r$. In many situations, a bounded number of
\emph{adjacent dyadic systems} $\D^\alpha$,
$$
\mathscr{D}^\alpha \coloneqq \big\{2^{-k}([0,1)^d+m+(-1)^k \tfrac{1}{3}\alpha):k\in \Z, m\in \Z^d\big\}, \quad \alpha\in \{0,1,2\}^d,
$$
can be used to overcome this problem:

\begin{lemma}[See {\cite[Lemma 2.5]{HLP}}]
\label{lem:aprox}
For any ball $B \coloneqq B(x,r) \subset \R^d$, there exists a cube $Q_B \in \mathscr{D}^\alpha$ for some $\alpha\in \{0,1,2\}^d$ such that $B \subset Q_B$ and $6r < \ell(Q_B)\le 12r$.
\end{lemma}

We note that \cite[Lemma 2.5]{HLP} is actually a stronger lemma than Lemma \ref{lem:aprox} above but for clarity, we use this formulation.

In the light of Lemma \ref{lem:aprox}, the collection $\D_0 \coloneqq \bigcup_{\alpha \in \{0,1,2\}^d} \D^\alpha$ can be seen as a countable approximation
of the collection of all balls in $\R^d$. It still satisfies essentially the properties 1) and 3) that we listed earlier but it satisfies the property 2) only
in various weaker forms. We slightly abuse the common terminology and say that $Q$ is a \emph{dyadic cube} if $Q \in \D_0$.

The adjacent dyadic systems $\D^\alpha$ satisfy also the following property, which will be useful for us later in this section.

\begin{lemma}
  \label{lem:dyadicCoveringMod}
  If $Q_0 \in \D_0$, then for any ball $B \coloneqq B(x,r) \subset Q_0$ there exists a cube $Q_B \in \D_0$
  such that $B \subset Q_B \subseteq Q_0$ and $\ell(Q_B) \le 12r$.
\end{lemma}

\begin{proof}
  We will only detail the proof for $d = 1$. The general case follows by considering the cube that contains the ball, and repeating the one-dimensional consideration for each of its side intervals in every coordinate direction.

  We may assume that $r < \frac{1}{12}\ell{(Q_0)}$ since otherwise we can simply choose $Q_B = Q_0$. Let $k \ge 1$ be the unique integer such that $6r < 2^{-k} \ell(Q_0) \le 12r$
  and let us look at the dyadic descendants of $Q_0$ of side length $2^{-k} \ell(Q_0)$. Since $B \subset Q_0$, we know that there exists at least one such descendant interval $I$ that $B \cap I \neq \emptyset$.
  If $B \subset I$, we can simply choose $Q_B = I$. Thus, we may assume that $B \not\subset I$.

  Since $2r < \frac{1}{3} \ell(I)$, the ball $B$ can only intersect the left or right third of the interval $I$. By symmetry, we can assume that
  the ball $B$ intersects the left third of the interval $I$. Then, by shifting the interval $I$ one third of its length to
  left, we can cover the ball $B$. Let $I_s$ be this shifted interval.
  By the definition of the collections $\D^\alpha$, we know that $I_s \in \D_0$. Since $B \subset Q_0$ and $B \not\subset I$, we know that
  there exists another dyadic descendant $J$ of $Q_0$ of side length $\ell(I)$ on the left side of $I$. Then, $I_s \subset I \cup J \subset Q_0$. In particular,
  $B \subset I_s \subset Q_0$ and we can set $Q_B = I_s$.
\end{proof}

For Lacey's domination theorem, the notion of sparse operators is crucial. Let $\mathscr{S}^\alpha \subset \D^\alpha$. Then we say that the operator
$\A_{\mathscr{S}^\alpha}$ is \emph{sparse} if
$$
  \A_{\mathscr{S}^\alpha} f(x) = \sum_{Q\in \mathscr{S}^{\alpha}}1_Q \ave{|f|}_Q
$$
and the collection $\mathscr{S}^\alpha$ satisfies the \emph{sparseness condition}: for each $Q\in \mathscr{S}^{\alpha}$ we have
$$
\Big|\bigcup_{\substack{Q'\in \mathscr{S}^{\alpha}\\Q'\subsetneq Q}}Q'\Big|\le \frac12|Q|.
$$
The sparseness condition is equivalent with a suitable \emph{Carleson condition}, see \cite[Section 6]{Ler-Naz}. We will use the notation
$\A_{\Ss}$ also for other types of collections $\Ss \subset \D_0$.

\subsection{Localized maximal truncations and truncated maximal operators}

Let $T$ be a Calder\'on-Zygmund operator with Dini-continuous kernel. For every cube $P\subset \R^d$, we define the \emph{$P$-localized maximal truncation of $T$} as the
operator $T_{\sharp,P}$,
$$
T_{\sharp,P}f(x):=\sup_{0<\eps<\delta< \frac{1}{2} \cdot \dist(x,\partial P)}\big|T_{\eps,\delta}f(x)\big| 1_P(x),
$$
where $T_{\eps,\delta}$ is another truncation operator given by
$$
T_{\eps,\delta}f(x):=\int_{\eps<|x-y|<\delta}K(x,y)f(y)\,dy.
$$
As an auxiliary operator, we also need the truncated centered Hardy-Littlewood maximal operator $M^c_{\eps,\delta}$,
$$
  M^c_{\eps,\delta}f(x) \coloneqq \sup_{\eps < r < \delta} \ave{|f|}_{B(x,r)}.
$$
The connection between the truncations $T_{\eps,\delta}$ and the maximal operator $M_{\eps,\delta}^c$ is formulated in the next lemma.

\begin{lemma}\label{lem:truncatedMO}
  Suppose that $|x - x'| < \frac{1}{2}\eps$. Then
  $$
    \left| T_{\eps, \delta} f(x) - T_{\eps, \delta}f(x') \right| \le c_d \left( C_K + \| \omega \|_{\Dini} \right) M^c_{\eps, 2\delta}f(x).
  $$
\end{lemma}

\begin{proof}
This is a straightforward calculation that we complete in several steps. First, let us write the left hand side of the inequality in a different form:
\begin{equation*}
\begin{split}
  \abs{T_{\eps,\delta}f(x)-T_{\eps,\delta}f(x')}
  &=\Babs{\int_{\eps<\abs{x-y}<\delta}K(x,y)f(y)\ud y-\int_{\eps<\abs{x'-y}<\delta}K(x',y)f(y)\ud y} \\
  &=\Big|\int_{\eps<\abs{x-y}<\delta}[K(x,y)-K(x',y)]f(y)\ud y \\
   &\qquad    +\Big(\int_{\eps<\abs{x-y}<\delta}-\int_{\eps<\abs{x'-y}<\delta}\Big)K(x',y)f(y)\ud y\Big| =:\abs{I+II}.
\end{split}
\end{equation*}
Then we can estimate terms $I$ and $II$ separately. For the term $I$, we use the smoothness of the kernel and the properties of $\omega$:
\begin{align*}
  \abs{I} &\le \int_{\eps<\abs{x-y}<\delta} \omega\left(\frac{\abs{x-x'}}{\abs{x-y}}\right)\frac{\abs{f(y)}}{\abs{x-y}^d}\ud y \\
          &\le \sum_{k:\eps\leq 2^k\eps<\delta} \int_{2^k\eps<\abs{x-y}\leq 2^{k+1}\eps}\omega\left(\frac{\abs{x-x'}}{2^k\eps}\right)\frac{\abs{f(y)}}{(2^k\eps)^d}\ud y \\
          &\le \sum_{k:\eps\leq 2^k\eps<\delta} \omega\left(\frac{\abs{x-x'}}{2^k\eps}\right)\int_{B(x,2^{k+1}\eps)}\frac{\abs{f(y)}}{(2^k\eps)^d}\ud y \\
          &\le \sum_{k=0}^\infty\omega\Big(\frac{\abs{x-x'}}{2^k\eps}\Big) c_d M_{\eps,2\delta}^cf(x) \\
          &\le c_d' M_{\eps,2\delta}^cf(x)\sum_{k=0}^\infty \int_{\abs{x-x'}/2^{k}\eps}^{\abs{x-x'}/2^{k-1}\eps}\omega(t)\frac{\ud t}{t} \\
          &=   c_d' M_{\eps,2\delta}^cf(x)\int_0^{2\abs{x-x'}/\eps}\omega(t)\frac{\ud t}{t}
           \le c_d' M_{\eps,2\delta}^cf(x)\int_0^{1}\omega(t)\frac{\ud t}{t}.
\end{align*}
For the term $II$, let us first notice that
\begin{equation*}
  II=II_{\eps}-II_{\delta},
\end{equation*}
where
\begin{equation*}
\begin{split}
  \qquad II_r
  &:=\Big(\int_{\abs{x-y}>r}-\int_{\abs{x'-y}>r}\Big)K(x',y)f(y)\ud y \\
  &=\Big(\int_{\abs{x-y}>r\geq\abs{x'-y}}-\int_{\abs{x'-y}>r\geq\abs{x-y}}\Big)K(x',y)f(y)\ud y.
\end{split}
\end{equation*}
Since $\abs{x-x'}\leq\frac12\eps\leq\frac12r$, in the first integral we have
$\abs{x-y}\leq\abs{x-x'}+\abs{x'-y}<\frac12 r+r\leq\frac32 r$, and in the second,
$\abs{x-y}\geq\abs{x'-y}-\abs{x'-x}>r-\frac12 r=\frac12 r$, so that $\frac12 r\leq \abs{x-y}\leq\frac32 r$ throughout.
By symmetry, we also have that  $\frac12 r\leq \abs{x'-y}\leq\frac32 r$. Thus, the size estimate of the kernel $K$ gives us
\begin{equation*}
\begin{split}
  \abs{II_r}
  &\leq\int_{\frac12 r\leq\abs{x'-y},\abs{x-y}\leq\frac 32 r }\frac{C_K}{\abs{x'-y}^d}\abs{f(y)}\ud y \\
  &\leq c_d C_K \fint_{B(x,\frac32 r)}\abs{f(y)}\ud y
  \leq c_d C_K M^c_{\eps,2\delta}f(x),
\end{split}
\end{equation*}
which proves the claim.
\end{proof}

\subsection{Lacey's domination theorem revisited}

In this section we will prove the following quantitative version of Lacey's pointwise domination theorem.

\begin{theorem}[Quantitative pointwise domination]
\label{thm:domination}
Let $T$ be a Calder\'on--Zygmund operator with Dini-continuous kernel. Then for any compactly supported function $f\in L^1(\R^d)$ there exist
sparse collections $\Ss^\alpha \subseteq \D^\alpha$, $\alpha \in \{0,1,2\}^d$, such that
\begin{equation}\label{ineq:pointwiseDomin}
  T_{\sharp}f(x) \ \le \ c_d\big(\|T\|_{L^2\to L^2}+C_K+\|\omega\|_{\Dini}\big)\sum_{\alpha \in \{0,1,2\}^d} \mathcal{A}_{ \mathscr{S}^{\alpha}}f(x)
\end{equation}
for almost every $x \in \R^d$, where the constant $c_d$ depends only on the dimension.
\end{theorem}

Note that Theorem \ref{thm:main1} is an immediate corollary of Theorem~\ref{thm:domination}, in combination with the following, by now well known estimate from \cite{HL} (see also \cite{HLi}).

\begin{theorem}
\label{thm:ApAinfty}
Let $1<p<\infty$ and $w\in A_p$. Then
$$
\|\mathcal{A}_{ \mathscr{S}^{\alpha}}\|_{L^p(w)\to L^p(w)}\le c_p \{w\}_{A_p}.
$$
\end{theorem}

\begin{remark}
\label{rem:CZO}
The original Calder\'on--Zygmund operator can be estimated in terms of the maximal truncations by $\abs{Tf(x)}\leq T_{\sharp}f(x)+\Norm{T}{L^2\to L^2}\abs{f(x)}$, cf. \cite[I.7.2]{Stein:HA}. (The case of the identity operator $T=I$, which is a Calder\'on--Zygmund operator with zero kernel, shows that the second term cannot be omitted.) Since $A_p$-weighted estimates for the second term (in effect, the identity operator) are trivial, Theorem \ref{thm:main1} remains true for the linear $T$ in place of $T_{\sharp}$, under the same assumptions.
\end{remark}

The novelty in both Theorems \ref{thm:main1} and \ref{thm:domination} is the explicit constant on the right. The main tool for the proof of Theorem \ref{thm:domination} is the following elaboration of Lacey's recursion lemma \cite[Lemma 4.7]{Lacey}, again with the same explicit constant. As mentioned in the Introduction, obtaining this explicit constant requires some nontrivial twist in the argument, and we provide the full details.

\begin{lemma}
\label{lem:recursion}
Let $f$ be an integrable function. Then for every $Q_0 \in \D_0$, there exists a collection $\Qs(Q_0)$ of dyadic cubes $Q\subset Q_0$ such that the following three conditions hold:
\begin{enumerate}
  \item \label{recursion:prop1} $\sum_{Q\in \Qs(Q_0)}|Q|< \eps_d |Q_0|$; (small size)
  \item \label{recursion:prop2} if $Q'\subset Q$ and $Q',Q\in \Qs(Q_0)$, then $Q'=Q$; (maximal with respect to inclusion)
  \item \label{recursion:prop3} we have
                          \begin{equation*}
                            T_{\sharp,Q_0}f\le C_T^0 \ave{|f|}_{Q_0}1_{Q_0}+\max_{Q\in \Qs(Q_0)}T_{\sharp,Q}f,
                          \end{equation*}
\end{enumerate}
where $C_T^0 \coloneqq c_d\big(\|T\|_{L^2\to L^2}+C_K+\|\omega\|_{\Dini}\big)$ and
$\eps_d$ can be taken as small as desired, at the cost of choosing a large enough $c_d$.
\end{lemma}

\begin{proof}
The idea of the proof is to show that for any constant $C_T^0 > 0$ we can cover the set $E_0$,
\begin{equation*}
  E_0 \coloneqq \{x\in Q_0: T_{\sharp,Q_0}f(x)>C_T^0\ave{\abs{f}}_{Q_0}\},
\end{equation*}
with countably many cubes $Q_i \in \D_0$ that satisfy conditions \eqref{recursion:prop2} and \eqref{recursion:prop3} and if the constant
$C^T_0$ is of the form $c_d\big(\|T\|_{L^2\to L^2}+C_K+\|\omega\|_{\Dini}\big)$, then the cubes also satisfy condition \eqref{recursion:prop1}.

Let $x \in E_0$. Since the function $(\eps,\delta) \mapsto T_{\eps,\delta}f(x)$ is continuous, we can choose such radii $0<\sigma_x<\tau_x\leq \frac{1}{2} \cdot \dist(x,\partial Q_0)$ that
\begin{equation*}
  \abs{T_{\sigma_x,\tau_x}f(x)}\geq C_T^0\ave{\abs{f}}_{Q_0}
\end{equation*}
and
\begin{equation*}
  \abs{T_{\sigma,\tau}f(x)}\leq C_T^0\ave{\abs{f}}_{Q_0}\qquad\text{if}\quad \sigma_x\leq\sigma\leq \tau\leq \frac{1}{2} \cdot \dist(x,Q_0).
\end{equation*}
For simplicity, we drop the conditions $\eps>0$ and $\delta\leq \frac{1}{2} \cdot \dist(x,Q_0)$ from the notation. Now the maximality of $\sigma_x$ implies the following:
\begin{align*}
  T_{\sharp,Q_0}f(x) &=\sup_{\eps\leq\delta}\abs{T_{\eps,\delta}f(x)} \\
                     &=\sup_{\eps\leq\delta\leq \sigma_x}\abs{T_{\eps,\delta}f(x)} \vee\sup_{\sigma_x\leq\eps\leq\delta}\abs{T_{\eps,\delta}f(x)}
                       \vee \sup_{\eps\leq \sigma_x\leq \delta}\abs{T_{\eps,\delta}f(x)} \\
                     & =:I\vee II\vee III,
\end{align*}
where
\begin{equation*}
  III=\sup_{\eps\leq \sigma_x\leq\delta}\abs{T_{\eps,\sigma_x}f(x)+T_{\sigma_x,\delta}f(x)}
  \leq I+II,
\end{equation*}
and $II\leq C_T^0\ave{\abs{f}}_{Q_0}$ by definition. So altogether we find that
\begin{equation}\label{eq:localBound}
  T_{\sharp,Q_0}f(x)\leq \sup_{\eps\leq\delta\leq \sigma_x}\abs{T_{\eps,\delta}f(x)}+C^0_T\ave{\abs{f}}_{Q_0}\qquad\forall x\in E_0,
\end{equation}
which is a preliminary version of the pointwise domination result we are proving. Now we can use Lemma \ref{lem:dyadicCoveringMod} to
get from the preliminary version to the desired estimate. Since $B(x, 2\sigma_x) \subset Q_0$ for every $x \in E_0$, there
exists a cube $Q_x \in \D_0$ such that $B(x, 2 \sigma_x) \subset Q_x \subset Q_0$ and $\ell(Q_x) \le 12 \cdot 2 \sigma_x$
for every $x \in E_0$. Let $(Q_i)_i$ be the sequence of such cubes $Q_x$ that are maximal with respect to inclusion,
that is, for each $Q_i$ there does not exist $R \in \D_0$ such that $Q_i\subsetneq R \subseteq Q_0$. Then for every $x \in E_0$ we have
\begin{align*}
  T_{\sharp,Q_0} f(x) &\overset{\eqref{eq:localBound}}{\le} \sup_{0 < \eps \le \delta \le \sigma_x} \left| T_{\eps,\delta} f(x) \right| + C^0_T \ave{|f|}_{Q_0} \\
                      &\le\sup_{0 < \eps \le \delta \le \frac{1}{2} \cdot \dist(x, \partial Q_x)} \left| T_{\eps,\delta} f(x) \right| + C^0_T \ave{|f|}_{Q_0} \\
                      &= T_{\sharp,Q_x} f(x) + C^0_T \ave{|f|}_{Q_0} \\
                      &\leq \max_i T_{\sharp,Q_i} f(x) + C^0_T \ave{|f|}_{Q_0}
\end{align*}
and for every $x \in Q_0 \setminus E_0$ we have $T_{\sharp,Q_0} f(x) \le C^0_T \ave{|f|}_{Q_0}$ by definition. Thus, the cubes $Q_i$ satisfy Lacey's conditions (2) and (3)
and to complete the proof, we only need to show that with a suitable choice of $C^0_T$ the cubes also satisfy property \eqref{recursion:prop1}.

Let us split the set $E_0$ into two parts:
\begin{equation*}
  E_1:=\{x\in E_0: M_{\sigma_x,2\tau_x}f(x)\leq C^1_T\ave{\abs{f}}_{Q_0}\},\qquad E_2:=E_0\setminus E_1,
\end{equation*}
where $C^1_T$ is a constant whose value we will fix in the next step. Then, for $x\in E_1$ and $x'\in B(x,\frac12\sigma_x)$, we have
\begin{align*}
  \abs{T_{\sigma_x\tau_x}f(x')-T_{\sigma_x\tau_x}f(x)} &\overset{\ref{lem:truncatedMO}}{\le} c_d(C_K + \Norm{\omega}{\operatorname{Dini}})M_{\sigma_x,2\tau_x}^c f(x) \\
                                                       &\le c_d(C_K + \Norm{\omega}{\operatorname{Dini}})C_T^1\ave{\abs{f}}_{Q_0} \\
                                                       &=\frac12 C_T^0\ave{\abs{f}}_{Q_0},
\end{align*}
provided that we choose
\begin{equation*}
  C_T^1:=\frac{C_T^0}{2c_d(C_K+\Norm{\omega}{\operatorname{Dini}})}.
\end{equation*}
Then, since $x \in E_1 \subseteq E_0$, it follows that
\begin{equation*}
  T_{\sharp}(1_{Q_0}f)(x') \ge \abs{T_{\sigma_x,\tau_x}f(x')} \ge \abs{T_{\sigma_x,\tau_x}f(x)}-\frac12 C_T^0 \ave{\abs{f}}_{Q_0} > \frac12 C_T^0\ave{\abs{f}}_{Q_0}
\end{equation*}
for all $x'\in B(x,\frac12\sigma_x)$. In particular,
\begin{equation*}
\begin{split}
  \left| \bigcup_{x\in E_1}B(x,\tfrac12\sigma_x) \right| &\leq \left| \{T_{\sharp}(1_{Q_0}f)>\tfrac12 C_T^0\ave{\abs{f}}_{Q_0}\} \right| \\
                                                         &\leq \frac{\Norm{T_{\sharp}}{L^1\to L^{1,\infty}}}{\tfrac12 C_T^0\ave{\abs{f}}_{Q_0}}\Norm{1_{Q_0}f}{L^1}
                                                          = \frac{2\Norm{T_{\sharp}}{L^1\to L^{1,\infty}}}{C_T^0}\abs{Q_0}
\end{split}
\end{equation*}
by the weak $L^1$ inequality of $T_\sharp$.

Let us then show that with this choice of $C^1_T$ and a suitable choice of $C_T^0$ the size of $E_2$ is controlled.
Let $x \in E_2$. By definition, we can choose some $\rho_x\in [\sigma_x,2\tau_x]$ such that
\begin{equation*}
  \fint_{B(x,\rho_x)}\abs{f(y)}\ud y>C_T^1\ave{\abs{f}}_{Q_0}.
\end{equation*}
Since $\tau_x \le \frac{1}{2} \cdot \dist(x,\partial Q_0)$, we know that $B(x,2\rho_x) \subset Q_0$. In particular,
\begin{equation*}
  M(1_{Q_0}f)(x')>C^1_T\ave{\abs{f}}_{Q_0}
\end{equation*}
for all $x'\in B(x,\rho_x)$, where $M$ is the noncentered Hardy-Littlewood maximal operator
$$
Mf(x):=\sup_{B\ni x}\fint_B|f|dx.
$$
 Thus
\begin{align*}
  \left| \bigcup_{x\in E_2}B(x,\frac12\sigma_x) \right| &\leq
  \left| \bigcup_{x\in E_2}B(x,\rho_x) \right| \leq\abs{\{M(1_{Q_0}f)>C_T^1\ave{\abs{f}}_{Q_0}\}} \\
                                               &\leq\frac{c_d}{C^1_T\ave{\abs{f}}_{Q_0}}\Norm{1_{Q_0}f}{L^1}
                                                =\frac{c_d'(C_K+\Norm{\omega}{\operatorname{Dini}})}{C_T^0}\abs{Q_0}.
\end{align*}
by the weak $L^1$ inequality of Hardy-Littlewood maximal operator.

Finally, let us combine all the previous calculations. For every maximal cube $Q_i$, let $x_i \in E_0$ be a point such that $Q_i = Q_{x_i}$. Then, since
$\ell(Q_x) \le 12 \cdot 2\sigma_x$ for each $x \in E_0$, we have $|Q_{x_i}| \le c_d |B(x_i, \tfrac{1}{2} \sigma_{x_i})|$ for every $i$. In particular,
since the cubes in the collection $\{ Q_{x_i} \colon Q_{x_i} \in \D^\alpha\}$ are pairwise disjoint for a fixed $\alpha \in \{0,1,2\}^d$, we get that
\begin{align*}
  \sum_i |Q_{x_i}| &= \sum_{\alpha \in \{0,1,2\}^d}
     \sum_{i:Q_{x_i}\in\mathscr{D}^\alpha } \left| Q_{x_i} \right| \\
   &  \leq c_d \sum_{\alpha \in \{0,1,2\}^d}     \sum_{i:Q_{x_i}\in\mathscr{D}^\alpha }
      \left| B(x_i,\tfrac12\sigma_{x_i}) \right| \\
   &  = c_d \sum_{\alpha \in \{0,1,2\}^d}
   \Big|  \bigcup_{i:Q_{x_i}\in\mathscr{D}^\alpha }
       B(x_i,\tfrac12\sigma_{x_i}) \Big| \\
                   &\le 3^d c_d \left( \Big| \bigcup_{x \in E_1} B(x, \tfrac{1}{2}\sigma_{x}) \Big| + \Big| \bigcup_{x \in E_2} B(x, \rho_x) \Big| \right) \\
                   &\le c_d' \frac{\| T_\sharp \|_{L^1 \to L^{1,\infty}} + C_K + \| \omega \|_{\Dini}}{C_T^0} |Q_0|.
\end{align*}
By Corollary \ref{cor:quantTrunc}, we know that $\| T_\sharp \|_{L^1 \to L^{1,\infty}} \le c_d( \|T\|_{L^2 \to L^2} + C_K + \| \omega \|_{\Dini}).$ Hence, if
\begin{equation*}\label{eq:TheConstant}
  C^0_T = c_d(C_K+\Norm{\omega}{\operatorname{Dini}}+\Norm{T}{L^2\to L^2}),
\end{equation*}
then the cubes $Q_i$ satisfy property \eqref{recursion:prop1}.
\end{proof}

With the lemma at hand, Theorem~\ref{thm:domination} will follow by repeating the rest of Lacey's original proof, with the application of his Lemma 4.7 replaced by our Lemma~\ref{lem:recursion}. For completeness of the presentation and convenience of the reader, we also provide the details here:

\begin{proof}[Proof of Theorem \ref{thm:domination}]
  Let $f$ be a compactly supported integrable function and let $B$ be a ball such that $B \supseteq \supp f$. By Lemma \ref{lem:aprox}, there exists a dyadic cube $P_0 \in \D_0$
  such that $2B \subset P_0$. Our strategy is to construct a collection $\Ss_0 \subset \D_0$ and a nested sequence of collections
  $\Ss_n \subset \D_0$, $n = 1,2,\ldots$, such that the collections $\D^\alpha \cap \bigcup_{n=0}^\infty \Ss_n^\alpha$ are sparse, the operator $\A_{\Ss_0}$ satisfies the pointwise inequality
  \eqref{ineq:pointwiseDomin} for every $x \notin P_0$ and the operator $\A_{\Ss_*}$ satisfies \eqref{ineq:pointwiseDomin} for almost every $x \in P_0$, where $\Ss_* \coloneqq \bigcup_{n=1}^\infty \Ss_n$.
  We prove the theorem in three parts.

  \textbf{Part 1}: Construction of the collection $\Ss_0$. Let $\kappa_d$ be a dimensional constant such that
  $\dist(x,\partial (2 \kappa_d P_0)) \ge \text{diam}(P_0)$ for every $x \in P_0$, where $2\kappa_d P_0$ is the concentric enlargement
  of $P_0$ with side length $2\kappa_d \cdot \ell(P_0)$. Then we can use Lemma \ref{lem:aprox} to take a cube $P_1 \in \D_0$ such
  that $2 \kappa_d P_0 \subset P_1$. Since $\supp f \subseteq B$ and $\dist(x,y) \le \text{diam}(P_0) \le \dist(x,\partial P_1)$ for every $x \in P_0$ and $y \in B$, we have $T_{\sharp}f(x) \le T_{\sharp,P_1}f(x)$
  for every $x \in P_0$.

With this, we see that the construction of the collection $\Ss_0$ is simple. For every $i = 2,3,\ldots$, let $P_i \in \D_0$ be a dyadic cube given by Lemma \ref{lem:aprox} such that
  $2 P_{i-1} \subset P_i$. Then, since $\supp f \subset B$, we have $T_{\varepsilon,\delta} f(x) = 0$ for every $x \notin B$ and $\delta < \dist(x,B)$. Also,
  for $x \in P_{n+1} \setminus P_n$, it holds that $\ell(P_{n+1}) \le c_d \cdot \dist(x,B)$ by the construction of the cubes. Thus, for $x \in P_{n+1} \setminus P_n$, we have
  \begin{align*}
    T_\sharp f(x) &= \sup_{\varepsilon > 0} \left| \left( \int_{\varepsilon < |x-y| < \dist(x,B) \vee \varepsilon} + \int_{|x-y| > \dist(x,B) \vee \varepsilon} \right) K(x,y) f(y) \, dy \right| \\
                  &= \sup_{\eps > \dist(x,B)} \left| \int_{|x-y| > \eps} K(x,y) f(y) \, dy \right| \\
                  &\overset{\eqref{eq:CK}}{\le} C_K \sup_{\eps > \dist(x,B)} \int_{|x-y| > \eps} \frac{|f(y)|}{\varepsilon^d} \, dy \\
                  &\le c_d C_K \int_{\R^d} \frac{|f(y)|}{\ell(P_{n+1})^d} \, dy =  c_d C_K \ave{|f|}_{P_{n+1}} 1_{P_{n+1}}(x).
  \end{align*}
Thus, we can set $\Ss_0 = \{P_i \colon i = 1,2,\ldots\}$. We note that the collections $\Ss_0 \cap \D^\alpha$ are sparse by construction.

  \textbf{Part 2}: Construction of the collections $\Ss_n$, $n \ge 1$. From now on, we say that a cube in a given collection is \emph{maximal} if it is maximal
  with respect to inclusion and we say that it is \emph{size-maximal} if it is maximal with respect to side length.

  Let $\Pc_1$ be the collection $\Qs(P_1)$ given by Lemma \ref{lem:recursion} and denote $\Pc_1^* \coloneqq \{ Q \in \Pc_1 \colon \ell(Q) = \max\{\ell(Q') \colon Q' \in \Pc_1\}\}$.
  Recursively, we set
  \begin{align*}
    \Rc_{n+1} &\coloneqq \left( \Pc_n \setminus \Pc_n^* \right) \cup \bigcup_{Q \in \Pc_n^*} \Qs(Q), \\
    \Pc_{n+1} &\coloneqq \text{maximal cubes of } \Rc_{n+1}, \\
    \Pc_{n+1}^* &\coloneqq \text{size-maximal cubes of } \Pc_{n+1}.
  \end{align*}
  Using the collections $\Pc_n^*$ we can define the collections $\Ss_n$: we set
$$
    \Ss_1 \coloneqq \{P_1\},\qquad \Ss_{n+1} \coloneqq \Ss_n \cup \Pc_n^*.
$$
  Thus, we start with the collection $\Pc_1 = \Qs(P_1)$ and pick the size-maximal cubes $Q_i \in \Qs(P_1)$ to form $\Pc_1^*$. Then, we add these cubes $Q_i$ to the collection
  $\Ss_1^\alpha$ to form the collection $\Ss_2^\alpha$ and apply Lemma \ref{lem:recursion} for each of them to get the collections $\Qs(Q_i)$. We add these ``new'' cubes to
  the collection $\Qs(P_1) \setminus \{Q_i\}_i$ which gives us the collection $\Rc_2$. Then, we form the collection $\Pc_2$ by removing the cubes that are not maximal and start over.
  This way the cubes in $\Ss_{n+1} \setminus \Ss_n$ have strictly smaller side length than all the cubes in $\Ss_n$ for every $n \in \N$.

  We now claim that
  \begin{equation}
    \label{induction:bound} T_{\sharp, P_1} f \le C_T^0 \mathcal{A}_{\Ss_n}f + \max_{Q \in \Pc_n} T_{\sharp,Q}f
  \end{equation}
  for every $n = 1,2,\ldots$, where $C_T^0 = c_d\big(\|T\|_{L^2\to L^2}+C_K+\|\omega\|_{\Dini}\big)$ as in the previous proof. For $n=1$ the claim is true by Lemma
  \ref{lem:recursion}. Let us then assume that the claim holds for $n = k$. Then
  \begin{align*}
    \max_{Q \in \Pc_k} T_{\sharp,Q}f &= \max\left\{ \max_{Q \in \Pc_k \setminus \Pc_k^*} T_{\sharp,Q}f, \max_{Q \in \Pc_k^*} T_{\sharp,Q}f \right\} \\
                                     &\overset{\ref{lem:recursion}}{\le} \max\left\{ \max_{Q \in \Pc_k \setminus \Pc_k^*} T_{\sharp,Q}f, \max_{Q \in \Pc_k^*} \left\{ C_T^0 \ave{|f|}_Q 1_Q + \max_{Q' \in \Qs(Q)} T_{\sharp,Q'}f \right\} \right\} \\
                                     &\le \max\left\{ \max_{Q \in \Pc_k \setminus \Pc_k^*} T_{\sharp,Q}f, \max_{Q \in \Pc_k^*} \, \max_{Q' \in \Qs(Q)} T_{\sharp,Q'}f \right\} + C_T^0 \sum_{Q \in \Pc_k^*} \ave{|f|}_Q 1_Q,
  \end{align*}
  and hence
$$
    T_{\sharp,P_1}f  \le  C_T^0 \mathcal{A}_{\Ss_k}f + \max_{Q \in \Pc_k} T_{\sharp,Q}f  \le  C_T^0 \mathcal{A}_{\Ss_{k+1}}f + \max_{Q \in \Pc_{k+1}} T_{\sharp,Q} f.
$$
  by the following fact: if $Q' \subseteq Q$ then $T_{\sharp,Q'}f(x) \le T_{\sharp,Q}f(x)$.

  The a.e. pointwise bound \eqref{ineq:pointwiseDomin} follows from \eqref{induction:bound} in the following way. Let us fix $n \in \N$ and denote
  $\Tc_{n,k} \coloneqq \{Q \in \Pc_n \colon \ell(Q) \le 2^{-k} \ell(Q'):Q'\in  \Pc_{n}^*\}$ for every $k \in \N$, i.e. we get the collection
  $\Tc_{n,k}$ from $\Pc_n$ by taking away $k$ generations of size-maximal cubes.
  Then, since we have $\sum_{Q \in \Pc_n} |Q| < \infty$, we can choose a large integer $k_n \in \N$ such that $\sum_{Q \in \Tc_{n,k_n}} |Q| \le \varepsilon_d \sum_{Q \in \Pc_n} |Q|$.
  Since it holds that $\Pc_{n + k_n} \subseteq \Tc_{n,k_n} \cup \bigcup_{Q \in \Pc_n \setminus \Tc_{n,k_n}} \Qs(Q)$, we get
  \begin{align*}
    \sum_{Q \in \Pc_{n + k_n}} |Q| &\le \sum_{Q \in \Tc_{n,k_n}} |Q| + \sum_{Q \in \Pc_n \setminus \Tc_{n,k_n}} \sum_{Q' \in \Qs(Q)} |Q| \\
                                   &\overset{\ref{lem:recursion}}{\le} \varepsilon_d \sum_{Q \in \Pc_n} |Q| + \sum_{Q \in \Pc_n \setminus \Tc_{n,k_n}} \varepsilon_d|Q| \\
                                   &\le 2 \varepsilon_d \sum_{Q \in \Pc_n} |Q| \le \frac{1}{2} \sum_{Q \in \Pc_n} |Q|.
  \end{align*}
  Thus, we need to apply the recursion to the cubes of $\Pc_n$ only finitely many times to halve the mass of the cubes.
  In particular, $\lim_{n \to \infty} \sum_{Q \in \Pc_n} |Q| = 0$ and hence, for almost every $x \in P_1$ there exists an integer $n_x \in \N$ such that
  $x \notin \bigcup_{Q \in \Pc_{n_x}} Q$. This gives us the a.e. pointwise bound \eqref{ineq:pointwiseDomin}.

  \textbf{Part 3}: Sparseness of the collections $\D^\alpha \cap \bigcup_{n=0}^\infty \Ss_n$. Let us recall the notation $\Ss_* = \bigcup_{n=1}^\infty \Ss_n$.
  To prove the sparseness of the collections $\D^\alpha \cap \Ss_*$, we will prove a stronger claim. For this, we need some definitions and notation:
  \begin{enumerate}
    \item[1)] We say that a cube $Q' \in \Ss_*$ is a \emph{$\Ss_*$-child} of a cube $Q \in \Ss_*$ (denote $Q' \in \ch_{\Ss_*}(Q)$) if $Q' \subsetneq Q$ and there does not exist a cube $Q'' \in \Ss_*$ such that
              $Q' \subsetneq Q'' \subsetneq Q$.
    \item[2)] We denote $n(Q) \coloneqq \min\{n \in \N \colon Q \in \Ss_n\}$ for every $Q \in \Ss_*$.
  \end{enumerate}
  Recall that if $Q \in \Pc_{k+1}$, then $Q \in \Pc_k \setminus \Pc_k^*$ or $Q \in \Qs(R)$ for some $R \in \Pc_k^*$.
  \begin{enumerate}
    \item[3)] We denote $m(Q) \coloneqq \min\{m \in \N \colon Q \in \Pc_k \text{ for every } k = m,\ldots,n(Q)-1\}$.
    \item[4)] We say that a cube $Q \in \Ss_*$ is a \emph{$\sharp$-parent} of a cube $Q' \in \Ss_*$ (denote $Q = \widehat{Q'}$) if $Q \in \Pc_{m(Q')-1}^*$ and $Q' \in \Qs(Q)$.
              We note that $n(\widehat{Q'}) = m(Q')$ and that $\widehat{Q'}$ may not be unique but we fix some $\widehat{Q'}$ for each $Q'$.
  \end{enumerate}
  Our goal is to show that the collection $\Ss_*$ satisfies a sparseness-type condition $| \bigcup_{Q' \in \, \ch_{\Ss_*}(Q)} Q'| \le C \cdot \varepsilon_d |Q|$.

  The following property of $\sharp$-parents is crucial:
  \begin{equation}
    \label{properties:sharp_parents} \text{if } Q' \in \ch_{\Ss_*}(Q) \text{ and } \widehat{Q'} \neq Q, \quad \text{ then } \ell(\widehat{Q'}) \le \ell(Q) \text{ and } \widehat{Q'} \not\subset Q.
  \end{equation}
  The property $\ell(\widehat{Q'}) \le \ell(Q)$ follows from the simple observation that if $\ell(\widehat{Q'}) > \ell(Q)$, then also $n(Q) \ge n(\widehat{Q'}) + 1$.
  Since $Q' \in \Pc_{n(\widehat{Q'})}$, this would then imply that $Q',Q \in \Pc_n$ for some same $n$. Since $Q' \subsetneq Q$, this is impossible by the maximality of the
  collections $\Pc_n$. The property $\widehat{Q'} \not\subset Q$ follows directly from the maximality of the
  $\Ss_*$-children of $Q$: otherwise we would have $Q' \subsetneq \widehat{Q'} \subsetneq Q$.

  By the property \eqref{properties:sharp_parents} we know the following: if $Q' \in \ch_{\Ss_*}(Q)$, then either $Q' \in \Qs(Q)$ or $Q' \in \Qs(S)$ for some $S \in \Ss_*$ such that $\ell(S) \le \ell(Q)$ and
  $S \not\subset Q$. In either case, the following statement holds: there exists a cube $S \in \Ss_*$ such that $Q' \in \Qs(S)$, $\ell(S) = 2^{-n} \ell(Q)$ and $S \subset (1 + 2^{-n+1})Q \setminus (1 - 2^{-n+1})Q$
  for some $n \in \{0,1,\ldots\}$, where $cQ$ is the cube with same center point as $Q$ with side length $\ell(cQ) = c \cdot \ell(Q)$ and $cQ = \emptyset$ if $c \le 0$. Let us denote
  $$
    \Bc_n(Q) \coloneqq \{ S \in \Ss_* \colon \ell(S) = 2^{-n}\cdot\ell(Q), \ S \subset (1 + 2^{-n+1})Q \setminus (1 - 2^{-n+1})Q\}
  $$
  for every $n = 0,1,\ldots$. Then, since the cubes of the collection $\Bc_n(Q) \cap \D^\alpha$ are disjoint and contained in $(1 + 2^{-n+1})Q \setminus (1 - 2^{-n+1})Q$ for every
  $\alpha \in \{0,1,2\}^d$, we know that
  \begin{equation}
    \label{eq:measureSBn} \sum_{S \in \Bc_n(Q)} |S| \le 3^d \left| (1 + 2^{-n+1})Q \setminus (1 - 2^{-n+1})Q \right|
  \end{equation}
  for every $n = 0,1,\ldots$. Thus, if $Q' \in \ch_{\Ss_*}(Q)$, then $Q \in \Qs(S)$ for some $S \in \Bc_n(Q)$ and $n \in \{0,1,\ldots\}$. Hence, for every $Q \in \Ss$ we have
  \begin{align*}
    \Big|\bigcup_{\substack{Q'\in \Ss_* \\Q'\subsetneq Q}}Q'\Big| &\le \sum_{n=0}^\infty \, \sum_{S \in \Bc_n(Q)} \, \sum_{Q' \in \Qs(S)} |Q'|
                                                                \, \overset{\ref{lem:recursion}}{\le} \ \, \varepsilon_d \cdot \sum_{n=1}^\infty \sum_{S \in \Bc_n(Q)} |S| + \varepsilon_d \sum_{S \in \Bc_0(Q)} |S| \\
                                                                &\overset{\eqref{eq:measureSBn}}{\le} 3^d \cdot \varepsilon_d \sum_{n=1}^\infty \left| (1 + 2^{-n+1})Q \setminus (1 - 2^{-n+1})Q \right| + 3^{d} \cdot \varepsilon_d |3Q| \\
                                                                &\le 3^d \cdot \varepsilon_d |Q| \sum_{n=1}^\infty 2^{-n + 2} \cdot 2^d + 3^{2d} \cdot \varepsilon_d |Q|
                                                                \, \le \, 5\cdot 3^{2d} \cdot \varepsilon_d |Q|.
  \end{align*}
  In particular, if we choose $\varepsilon_d$ to be small enough, the collections $\Ss_* \cap \D^\alpha$ are sparse.

  Finally, we note that the collections $\Ss^\alpha \coloneqq \D^\alpha \cap (\Ss_* \cup \Ss_0)$ are sparse since adding the large cubes $P_2, P_3, \ldots$ to the corresponding collections
  $\Ss_* \cap \D^\alpha$ does not affect the sparseness of those collections. This completes the proof.
\end{proof}

\section{Rough homogeneous Calder\'on--Zygmund operators}

In this section, we prove Theorem \ref{thm:main2}. The techniques are originally developed in \cite{Duo-RdF} and \cite{Watson}, but we adapt them and modify them in order to get the dependence of the results in terms of the $A_2$ characteristic of the weight. We will use Theorem \ref{thm:main1} as a black box. Also a clever choice in the expression of our rough operators will refine such dependence.

To begin with, the proof of Theorem \ref{thm:main2} requires some ingredients that are shown in the subsequent subsections.

Recall the definition of the operator $T_{\Omega}$ given in the introduction. It can be written as
\begin{equation}
\label{eq:TOmega}
T_{\Omega}=\sum_{k\in \Z}T_kf=\sum_{k\in \Z}K_k\ast f, \quad K_k=\frac{\Omega(x')}{|x|^d}\chi_{2^k<|x|<2^{k+1}}.
\end{equation}

The following lemma is well-known, and the proof can be found in \cite{Duo-RdF}.

\begin{lemma}
\label{lem:estimateKk}
The following inequality holds
$$
|\widehat{K}_k(\xi)|\le c_d \|\Omega\|_{L^\infty}\min( |2^k\xi|^\alpha,|2^k\xi|^{-\alpha}),
$$
for some numerical $0<\alpha<1$ independent of $T_{\Omega}$ and $k$.
\end{lemma}

\begin{remark}
The estimates in Lemma \ref{lem:estimateKk} are still valid for $\Omega \in L^{p}(\Sp^{d-1})$, $p>1$, or even $\Omega \in L\log L(\Sp^{d-1})$. Nevertheless, in order to obtain weighted estimates we will require $\Omega \in L^{\infty}(\Sp^{d-1})$.
\end{remark}

We consider the following partition of unity. Let $\phi\in C_c^{\infty}(\R^d)$ be such that $\supp \phi\subset\{x:|x|\le \frac{1}{100}\}$  and $\int\phi\,dx=1$, and so that $\widehat{\phi}\in \mathcal{S}(\R^d)$. Let us also define $\psi$ by $\widehat{\psi}(\xi)=\widehat{\phi}(\xi)-\widehat{\phi}(2\xi)$. Then, with this choice of $\psi$, it follows that $\int\psi\,dx=0$. We write $\phi_j(x)=\frac{1}{2^{jd}}\phi\big(\frac{x}{2^j}\big)$, and $\psi_j(x)=\frac{1}{2^{jd}}\psi\big(\frac{x}{2^j}\big)$.

We now define the partial sum operators $S_j$ by $S_j(f)=f\ast \phi_j$. Their differences are given by
\begin{equation}
\label{eq:difference}
S_j(f)-S_{j+1}(f)=f\ast\psi_j.
\end{equation}
Since $S_j f\to 0$ as $j\to-\infty$, for any sequence of integer numbers $\{N(j)\}_{j=0}^{\infty}$, with $0=N(0)<N(1)<\cdots<N(j)\to\infty$, we have the identity
$$
T_k=T_kS_k+\sum_{j=1}^{\infty}T_k(S_{k-N(j)}-S_{k-N(j-1)}).
$$
In this way, $T_{\Omega}=\sum_{j=0}^\infty\widetilde{T}_j=\sum_{j=0}^{\infty}\widetilde{T}_j^N$, where
\begin{equation}
\label{eq:T0}
\widetilde{T}_0:=\widetilde{T}_0^N:=\sum_{k\in\Z}T_kS_k,
\end{equation}
and, for $j\ge1$,
\begin{equation}
\begin{split}\label{eq:Tj}
\widetilde{T}_j &:=\sum_{k\in\Z}T_k(S_{k-j}-S_{k-(j-1)}), \\
  \widetilde{T}_j^N &:=\sum_{k\in\Z}T_k(S_{k-N(j)}-S_{k-N(j-1)})
  =\sum_{i=N(j-1)+1}^{N(j)}\widetilde{T}_i.
\end{split}
\end{equation}

\subsection{$L^2$ estimates for $\widetilde{T}_j^N$}

\begin{lemma}[Unweighted $L^2$ estimates for $\widetilde{T}_j^N$]
\label{lem:L2}
Let $\widetilde{T}_j$ and $\widetilde{T}_j^N$ be the operators as in \eqref{eq:T0} and \eqref{eq:Tj}. Then we have
\begin{align*}
  \|\widetilde{T}_jf\|_{L^2} &\le c_d \|\Omega\|_{L^\infty}2^{-\alpha j}\|f\|_{L^2}, \\
  \|\widetilde{T}_j^Nf\|_{L^2} &\le c_d \|\Omega\|_{L^\infty}2^{-\alpha N(j-1)}\|f\|_{L^2},
\end{align*}
for some numerical $0<\alpha<1$ independent of $T_{\Omega}$ and $j$.
\end{lemma}

\begin{proof}
Let us first consider $j\geq 1$.
From \eqref{eq:TOmega}, \eqref{eq:Tj} and \eqref{eq:difference}, we write
$$
\widehat{\widetilde{T}_jf}(\xi)=\sum_{k\in\Z}\widehat{K}_k(\xi)\widehat{\psi}(2^{k-j}\xi)\widehat{f}(\xi)=:m_{j}(\xi)\widehat{f}(\xi).
$$
 We will obtain a pointwise estimate for $m_{j}(\xi)$. Since $\widehat\psi\in\mathcal{S}(\R^d)$ and $\widehat{\psi}(0)=0$, we have $|\widehat{\psi}(\xi)|\le C\min(|\xi|,1)$, and hence
 \begin{equation}
 \label{eq:estimateHatpsi}
|\widehat{\psi}(2^{k-j}\xi)|\le C\min( |2^{k-j}\xi|,1).
 \end{equation}
Thus, \eqref{eq:estimateHatpsi} and Lemma \ref{lem:estimateKk} imply
\begin{equation}
\label{eq:estimatesMulti}
|\widehat{K}_k(\xi)||\widehat{\psi}(2^{k-N(j)}\xi)| \le c_d \|\Omega\|_{L^\infty}|2^k\xi|^{-\alpha}
  \min(|2^{k-j}\xi|,1),
\end{equation}
and hence
\begin{align*}
  |m_j(\xi)| &\leq c_d\|\Omega\|_{L^\infty}\Big(\sum_{k:2^k|\xi|\leq 2^j} 2^{-j}|2^k\xi|^{1-\alpha}
     +\sum_{k:2^k|\xi|\geq 2^j}|2^k\xi|^{-\alpha}\Big) \\
   &\leq c_d\|\Omega\|_{L^\infty}\big(2^{-j} 2^{j(1-\alpha)}+2^{-j\alpha})
   \leq c_d\|\Omega\|_{L^\infty}2^{-j\alpha}.
\end{align*}
The required $L^2$ inequality for $\widetilde{T}_j f$ then follows by Plancherel. To estimate $\widetilde{T}_j^N f$, we simply need to sum the geometric series $\sum_{i=N(j-1)+1}^{N(j)}2^{-\alpha i}$.

For $j=0$, we have $\widehat\phi(2^k\xi)$ in place of $\widehat\psi(2^{k-j}\xi)$ above. Then, in place of \eqref{eq:estimateHatpsi} and \eqref{eq:estimatesMulti}, we use simply $|\widehat\phi(2^k\xi)|\leq C$ and
\begin{equation*}
  |\widehat{K}_k(\xi)||\widehat{\phi}(2^{k}\xi)| \le c_d \|\Omega\|_{L^\infty}\min(|2^k\xi|,|2^k\xi|)^{-\alpha},
\end{equation*}
so that
\begin{equation*}
\begin{split}
  |m_0(\xi)|
  &:=\Big|\sum_{k\in\Z}\widehat{K}_k(\xi)||\widehat{\phi}(2^{k}\xi)\Big| \\
  &\leq c_d\|\Omega\|_{L^\infty}\Big(\sum_{k:2^k|\xi|\leq 1} |2^k\xi|^{\alpha}+\sum_{k:2^k|\xi|\geq 1}|2^k\xi|^{-\alpha}\Big)
    \leq c_d\|\Omega\|_{L^\infty}.
\end{split}
\end{equation*}
Again, Plancherel completes the $L^2$ estimate for $\tilde{T}_0 f=\tilde{T}_0^N f$.
\end{proof}

\subsection{Calder\'on--Zygmund theory of $\widetilde{T}_j^N$}

\begin{lemma}
\label{lem:modulus}
The operator $\widetilde{T}_j^N$ is a Calder\'on--Zygmund operator with
\begin{equation*}
  C_j^N:=C_{\widetilde{T}_j^N}\leq c_d\|\Omega\|_{L^\infty},\qquad
\omega_j^N(t):=\omega_{\widetilde{T}_j^N}(t)\le  c_d\|\Omega\|_{L^\infty} \min(1 , 2^{N(j)}t),
\end{equation*}
which satisfies
\begin{equation*}
   \int_0^1\omega_j^N(t)\frac{dt}{t}\leq c_d\|\Omega\|_{L^\infty}(1+N(j)).
\end{equation*}
\end{lemma}

\begin{proof}
We have already proved in Lemma \ref{lem:L2} that $\widetilde{T}_j^N$ is a bounded operator in $L^2$.
Recall the definition of $\widetilde{T}_j^N$ given in \eqref{eq:Tj}. In order to get the required estimates for the kernel of $\widetilde{T}_j^N$, we first study the kernel of each $T_kS_{k-N(j)}$. Let $x\in \R^n$. Since $\supp \phi\subset\{x:|x|\le \frac{1}{100}\}$, and passing to polar coordinates, then
\begin{align*}
|K_k\ast\phi_{k-N(j)}(x)|&=\bigg|\int_{\R^n}\frac{\Omega(y')}{|y|^d}1_{2^k<|y|<2^{k+1}}2^{-(k-N(j))d}\phi\Big(\frac{x-y}{2^{k-N(j)}}\Big)\,dy\bigg|\\
&\le c_d  \|\Omega\|_{L^\infty}\frac{1}{|x|^d}1_{2^{k-1}<|x|<3\cdot2^{k}} \int_{\R^n}2^{-(k-N(j))d}\bigg|\phi\Big(\frac{x-y}{2^{k-N(j)}}\Big)\bigg|\,dy\\
&\le c_d \|\Omega\|_{L^\infty}\frac{1}{|x|^d}1_{2^{k-1}<|x|<3\cdot2^{k}},
\end{align*}
so that
\begin{equation}
\label{eq:1st}
\sum_{k\in \Z}|K_k\ast\phi_{k-N(j)}(x)|\le c_d \|\Omega\|_{L^\infty}\sum_{k\in \Z}\frac{1}{|x|^d}1_{2^{k-1}<|x|<3\cdot2^{k}}\le c_d \frac{ \|\Omega\|_{L^\infty}}{|x|^d}.
\end{equation}
On the other hand, we compute the gradient. Again by taking into account the support of $\phi$ and passing to polar coordinates
\begin{align*}
|\nabla (K_k\ast&\phi_{k-N(j)})(x)|=\bigg|\int_{\R^n}\frac{\Omega(y')}{|y|^d}1_{2^k<|y|<2^{k+1}}2^{-(k-N(j))(d+1)}\nabla\phi\Big(\frac{x-y}{2^{k-N(j)}}\Big)\,dy\bigg|\\
&\le c_d  \|\Omega\|_{L^\infty}\frac{1}{|x|^d}1_{2^{k-1}<|x|<3\cdot2^{k}} \int_{\R^n}2^{-(k-N(j))(d+1)}\bigg|\nabla\phi\Big(\frac{x-y}{2^{k-N(j)}}\Big)\bigg|\,dy\\
&\le c_d \|\Omega\|_{L^\infty}\frac{1}{|x|^{d}}1_{2^{k-1}<|x|<3\cdot2^{k}}\frac{1}{2^{k-N(j)}}\\
&\le c_d  \|\Omega\|_{L^\infty}\frac{1}{|x|^{d+1}}1_{2^{k-1}<|x|<3\cdot2^{k}}2^{N(j)},
\end{align*}
thus
\begin{equation}
\begin{split}
\label{eq:2nd}
  \sum_{k\in \Z}|\nabla (K_k\ast\phi_{k-N(j)})(x)|
  &\le c_d \|\Omega\|_{L^\infty}\sum_{k\in \Z}\frac{1_{2^{k-1}<|x|<3\cdot2^{k}}}{|x|^{d+1}}2^{N(j)} \\
  &\le c_d \|\Omega\|_{L^\infty}\frac{2^{N(j)}}{|x|^{d+1}}.
\end{split}
\end{equation}
From the triangle inequality and $N(j-1)<N(j)$ it follows that the kernel $K_j^N:=\sum_{k\in\Z}K_k*(\phi_{k-N(j)}-\phi_{k-N(j-1)})$ of $\widetilde{T}_j^N$ satisfies the same estimate \eqref{eq:1st} and \eqref{eq:2nd}, i.e.
\begin{equation*}
\begin{split}
  |K_j^N(x,y)|=|K_j^N(x-y)| &\leq c_d\frac{\|\Omega\|_{L^\infty}}{|x-y|^d}, \\
 |\nabla K_j^N(x-y)| &\leq c_d\frac{\|\Omega \|_{L^\infty} }{|x-y|^{d+1}} 2^{N(j)}.
\end{split}
\end{equation*}
(For $j=0$, the subtraction is not even needed.) The first bound above is already the required estimate for $C_j^N$.
On the other hand, by the gradient estimate, for $|x-x'|\leq\frac12|x-y|$ we have 
\begin{equation*}
\begin{split}
\label{eq:smooth1}
 |K^{N}_j(x,y)-K^{N}_j(x',y)|&=|K^{N}_j(x-y)-K^{N}_j(x'-y)|\\
  &\leq|x-x'|\sup_{z\in[x,x']}|\nabla K_j^N(z-y)| \\
  &\leq|x-x'|\sup_{z\in[x,x']} c_d\frac{\|\Omega\|_{L^\infty}}{|z-y|^{d+1}}2^{N(j)} \\
  &\leq c_d\frac{\|\Omega\|_{L^\infty} }{|x-y|^d }2^{N(j)} \frac{|x-x'|}{|x-y|}.
\end{split}
\end{equation*}
From the triangle inequality we also have the easy bound
\begin{equation*}
   |K^{N}_j(x,y)-K^{N}_j(x',y)|\leq c_d\frac{\|\Omega\|_{L^\infty} }{|x-y|^d },
\end{equation*}
and combining the two estimates and symmetry,
$$
|K^{N}_j(x,y)-K^{N}_j(x',y)|+|K^{N}_j(y,x)-K^{N}_j(y,x')|\le c_d \,\omega_j^N\bigg(\frac{|x-x'|}{|x-y|}\bigg)\frac{1}{|x-y|^d},
$$
where
$$
\omega_j^N(t)\le c_d \|\Omega\|_{L^\infty} \min(1,2^{N(j)}t).
$$
The Dini norm of this function is estimated as
\begin{equation*}
\begin{split}
\int_0^1\omega_j(t)\frac{dt}{t}&\le c_d\|\Omega\|_{L^\infty} \Big(\int_0^{2^{-N(j)}}2^{N(j)}t
\frac{dt}{t}+\int_{2^{-N(j)}}^1\frac{dt}{t}\Big)\\
&=c_d\|\Omega\|_{L^\infty}(1+ \log 2^{N(j)})\le c_d\|\Omega\|_{L^\infty}(1+N(j)).
\end{split}
\end{equation*}
\end{proof}

In the following, we will prove a quantitative $L^p$ weighted inequality for the operators $\widetilde{T}_j^N$. 

\begin{lemma}
\label{lem:L2weighted}
Let $\widetilde{T}_j^N$ be the operators as in \eqref{eq:T0} and \eqref{eq:Tj}.  Let $1<p<\infty$. Then, for all $w\in A_p$, we have
$$
\|\widetilde{T}_j^N f\|_{L^p(w)}\le c_{d,p}\|\Omega\|_{L^\infty}\big(1+N(j)\big) \{w\}_{A_p}\|f\|_{L^p(w)},
$$
where $\alpha$ is a numerical constant independent of $T_\Omega$, $j$ and the function $N(\cdot)$.
\end{lemma}

\begin{proof}
By Theorem \ref{thm:main1} and Remark \ref{rem:CZO} for the first inequality below, and Lemma \ref{lem:L2} and Lemma \ref{lem:modulus} for the second one, we deduce that
 \begin{align*}
 \|\widetilde{T}_j^N\|_{L^p(w)}
 &\le  c_{d,p}\big(\|\widetilde{T}_j^N\|_{L^2\to L^2}+C_j^N+\|\omega_j^N\|_{\Dini}\big) \{w\}_{A_p}\|f\|_{L^p(w)}\\
 &\le c_{d,p}\big(2^{-\alpha N(j)}\|\Omega\|_{L^\infty}+\|\Omega\|_{L^\infty}+\|\Omega\|_{L^\infty}(1+N(j))\big)
 \{w\}_{A_p}\|f\|_{L^p(w)}\\
 &\le c_{d,p}\|\Omega\|_{L^\infty}\big(1+ N(j)\big)\{w\}_{A_p}\|f\|_{L^p(w)}.
 \end{align*}
\end{proof}

Finally, we will show that, from the unweighted $L^2$ estimate in Lemma \ref{lem:L2} and the unweighted $L^p$ estimate in Lemma \ref{lem:L2weighted} (i.e., the weighted estimate with $w(x)\equiv 1$), we can infer a good quantitative unweighted $L^p$ estimate for $\widetilde{T}_j$.

\begin{lemma}
\label{lem:Lpunweighted}
Let $\widetilde{T}_j^N$ be the operators as in \eqref{eq:T0} and \eqref{eq:Tj}.  Let $1<p<\infty$. Then,
$$
\|\widetilde{T}_j^N f\|_{L^p}\le c_{d,p}\|\Omega\|_{L^\infty}2^{-\alpha_p N(j-1)}(1+N(j)) \|f\|_{L^p},
$$
for some constant $\alpha_p$ independent of $T_\Omega$, $j$ and the function $N(\cdot)$.
\end{lemma}

\begin{proof}
First, assume that $p>2$, and take $q:=2p$, so that $2<p<q$. We have that $\frac{1}{p}=\frac{1-\theta}{2}+\frac{\theta}{q}$, for $0<\theta:=\frac{p-2}{p-1}<1$. Then, from Lemma \ref{lem:L2} and Lemma \ref{lem:L2weighted} with $w(x)\equiv1$, by complex interpolation, we get
\begin{equation*}
\begin{split}
 \|\widetilde{T}_j^N\|_{L^p\to L^p}
 &\leq \|\widetilde{T}_j^N\|_{L^2\to L^2}^{1-\theta}\|\widetilde{T}_j^N\|_{L^{2p}\to L^{2p}}^{\theta} \\
 &\leq (c_d\|\Omega\|_{L^\infty}2^{-\alpha N(j-1)})^{1-\theta}(c_{d,2p}\|\Omega\|_{L^\infty}(1+N(j)))^{\theta} \\
 &\leq c_{d,p}\|\Omega\|_{L^\infty}2^{-\alpha_p N(j-1)}(1+N(j)),
\end{split}
\end{equation*}
where $\alpha_p=\alpha(1-\theta)=\alpha/(p-1)$.

On the other hand, if $p<2$, let us take $q:=\frac{2p}{1+p}$, so that $1<q<p<2$. In this case, $\frac{1}{p}=\frac{1-\theta}{2}+\frac{\theta}{q}$, for $0<\theta:=2-p<1$. Once again, by interpolating between $L^2$ and $L^q$, we get an analogous $L^p$ estimate.
\end{proof}

\subsection{The Reverse H\"older Inequality}

We first need some recent results concerning the sharp Reverse H\"older Inequality (RHI). They are contained in the theorem below, where we combined \cite[Theorem 2.3]{HP} and \cite[Theorem 2.3]{HPR}

\begin{theorem}[\cite{HP,HPR}] 
\label{thm:sharpRHI}
There are dimensional constants $c_d,C_d$ with the following properties:
\begin{enumerate}[(a)]
\item \label{acon} Let $w\in A_{\infty}$. Then
$$
\fint_Q w^{1+\delta}\le 2\bigg(\fint_Q w\bigg)^{1+\delta},
$$
for any $\delta\in(0, c_d/[w]_{A_\infty}]$. 
\item \label{bcon} If a weight $w$ satisfies the RHI
$$
\bigg(\fint_Q w^{r}\bigg)^{1/r}\le K\fint_Qw,
$$
then $w\in A_\infty$ and $[w]_{A_{\infty}}\le C_d\cdot K\cdot r'$.
\end{enumerate}
\end{theorem}

As a consequence, we can infer the following corollary. Recall that $(w)_{A_p}:=\max([w]_{A_\infty},[w^{1-p'}]_{A_\infty})$.

\begin{corollary}
\label{cor:wepsilon}
Let $1<p<\infty$ and $w\in A_{p}$.
Then, there exists $c_d$ small enough such that for every $0<\delta \le c_d/(w)_{A_p}$, we have that $w^{1+\delta}\in A_p$ and
$$
[w^{1+\delta}]_{A_p}\le 4 [w]_{A_p}^{1+\delta}.
$$
\end{corollary}

\begin{proof}
By Theorem~\ref{thm:sharpRHI}, for a suitable $c_d$ and every $\delta\in(0,c_d/(w)_{A_p}]$, we have both
$$
\fint_Q w^{1+\delta}\le 2\bigg(\fint_Q w\bigg)^{1+\delta}
$$
and
$$
\fint_Q w^{(1-p')(1+\delta)}\le 2\bigg(\fint_Q w^{1-p'}\bigg)^{1+\delta}.
$$
Multiplying the two estimates and using the definition of $A_p$ gives the result.
\end{proof}

\begin{corollary}
\label{cor:Ainfty}
Let $w\in A_\infty$.
Then, there exists $c_d$ small enough such that for every $0<\delta \le c_d/[w]_{A_\infty}$, we have that $w^{1+\delta/2}\in A_\infty$ and
$$
[w^{1+\delta/2}]_{A_\infty}\le C_d [w]_{A_\infty}^{1+\delta/2}.
$$
\end{corollary}

\begin{proof}
Let $\delta_0:=c_d/[w]_{A_\infty}$, for a suitable $c_d$.
By \eqref{acon} in Theorem \ref{thm:sharpRHI}, $w$ satisfies the RHI
$$
\fint_Q w^{1+\delta_0}\le 2\bigg(\fint_Q w\bigg)^{1+\delta_0}.
$$
For $\delta\leq\delta_0$, the weight $w^{1+\delta/2}$ satisfies
\begin{equation*}
  \fint_Q (w^{(1+\delta/2)})^{(1+\delta_0)/(1+\delta/2)}
  \leq 2\Big(\fint_Q w\Big)^{1+\delta_0}
  \leq 2\Big(\fint_Q w^{1+\delta/2}\Big)^{(1+\delta_0)/(1+\delta/2)}.
\end{equation*}
By \eqref{bcon} of Theorem \ref{thm:sharpRHI}, $w^{1+\delta/2}\in A_\infty$, and
\begin{align*}
  [w^{1+\delta/2}]_{A_\infty}
  &\leq C_d\cdot 2\cdot\Big(\frac{1+\delta_0}{1+\delta/2}\Big)'
  =C_d\cdot 2\cdot\frac{1+\delta_0}{\delta_0-\delta/2} \\
  &\leq C_d\cdot\frac{8}{\delta_0}=C_d'\cdot[w]_{A_\infty}\leq C_d'\cdot[w]_{A_\infty}^{1+\delta/2}.
\end{align*}
\end{proof}

Finally, for our special weight characteristics $(w)_{A_p}$ and $\{w\}_{A_p}$, we have:

\begin{corollary}\label{cor:newWeights}
Let $1<p<\infty$ and $w\in A_{p}$.
Then, there exists $c_d$ small enough such that for every $0<\delta \le c_d/(w)_{A_p}$, we have that $w^{1+\delta/2}\in A_p$ and
\begin{equation*}
  (w^{1+\delta/2})_{A_p}\leq C_d(w)_{A_p}^{1+\delta/2},\qquad
  \{w^{1+\delta/2}\}_{A_p}\leq C_d\{w\}_{A_p}^{1+\delta/2}.
\end{equation*}
\end{corollary}

\begin{proof}
For the first bound, we apply Corollary~\ref{cor:Ainfty} to both $w\in A_\infty$ and $w^{1-p'}\in A_\infty$, observing that if $\delta\leq c_d/(w)_{A_p}$, then it satisfies both $\delta\leq c_d/[w]_{A_\infty}$ and $\delta\leq c_d/[w^{1-p'}]_{A_\infty}$. Hence
\begin{align*}
  (w^{1+\delta/2})_{A_p}
  &=\max\{[w^{1+\delta/2}]_{A_\infty},[w^{(1-p')(1+\delta/2)}]_{A_\infty}\} \\
  &\leq C_d\max\{[w]_{A_\infty}^{1+\delta/2},[w^{1-p'}]_{A_\infty}^{1+\delta/2}\}
  =C_d(w)_{A_p}^{1+\delta/2}.
\end{align*}
The other bound is similar, using in addition Corollary~\ref{cor:wepsilon}:
\begin{align*}
  \{w^{1+\delta/2}\}_{A_p}
  &=[w^{1+\delta/2}]_{A_p}^{1/p}\max\{[w^{1+\delta/2}]_{A_\infty}^{1/p'},[w^{(1-p')(1+\delta/2)}]_{A_\infty}^{1/p}\} \\
  &\leq (4 [w]_{A_p}^{1+\delta/2})^{1/p}
     \max\{(C_d[w]_{A_\infty}^{1+\delta/2})^{1/p'},(C_d[w^{1-p'}]_{A_\infty}^{(1+\delta/2)})^{1/p}\}  \\
   &\leq C_d'\{w\}_{A_p}^{1+\delta/2}.
\end{align*}
\end{proof}

\subsection{Proof of Theorem \ref{thm:main2}}

Let us denote $\eps:=\frac12 c_d/(w)_{A_p}$.
By Lemma \ref{lem:L2weighted} and Corollary \ref{cor:newWeights}, we have, for this choice of $\eps$,
\begin{equation*}
\begin{split}
  \|\tilde{T}_j^N \|_{L^p(w^{1+\eps})\to L^p(w^{1+\eps})}
  &\leq C_{d,p}\|\Omega\|_{L^\infty}(1+N(j))\{w^{1+\eps}\}_{A_p} \\
  &\leq C_{d,p}\|\Omega\|_{L^\infty}(1+N(j))\{w\}_{A_p}^{1+\eps}.
\end{split}
\end{equation*}
On the other hand, by Lemma~\ref{lem:Lpunweighted}, we also have
\begin{equation*}
  \|\tilde{T}_j^N \|_{L^p\to L^p}
  \leq C_{d,p}\|\Omega\|_{L^\infty}(1+N(j))2^{-\alpha_p N(j-1)}
\end{equation*}

Now we are in position to apply the interpolation theorem with change of measures by E. M. Stein and G. Weiss (see \cite[Th. 2.11]{Stein-Weiss}).
\begin{theorem}[Stein and Weiss]
\label{thm:interpolation}
Assume that $1\le p_0,p_1\le \infty$, that $w_0$ and $w_1$ are positive weights, and that $T$ is a sublinear operator satisfying
$$
T:L^{p_i}(w_i)\to L^{p_i}(w_i), \quad i=0,1,
$$
with quasi-norms $M_0$ and $M_1$, respectively. Then
$$
T:L^p(w)\to L^p(w),
$$
with quasi-norm $M\le M_0^{\lambda}M_1^{1-\lambda}$, where
$$
\frac{1}{p}=\frac{\lambda}{p_0}+\frac{(1-\lambda)}{p_1}, \quad w=w_0^{p\lambda/p_0}w_1^{p(1-\lambda)/p_1}.
$$
\end{theorem}

We apply Theorem \ref{thm:interpolation} to $T=\tilde{T}_j^N$ with $p_0=p_1=p$, $w_0=w^0=1$ and $w_1=w^{1+\eps}$ so that $\lambda=\eps/(1+\eps)$:
\begin{align*}
  \|\tilde{T}_j^N\|_{L^p(w)\to L^p(w)}
  &\leq \|\tilde{T}_j^N\|_{L^p\to L^p}^{\eps/(1+\eps)}\|\tilde{T}_j^N\|_{L^p(w^{1+\eps})\to L^p(w^{1+\eps})}^{1/(1+\eps)} \\
  &\leq C_{d,p}\|\Omega\|_{L^\infty}(1+N(j))2^{-\alpha_p N(j-1)\eps/(1+\eps)}\{w\}_{A_p} \\
  &\leq C_{d,p}\|\Omega\|_{L^\infty}(1+N(j))2^{-\alpha_{p,d} N(j-1)/(w)_{A_p}}\{w\}_{A_p}.
\end{align*}
Thus
\begin{align*}
  \|T_\Omega\|_{L^p(w)\to L^p(w)}
  &\leq\sum_{j=0}^\infty\|\tilde{T}_j^N\|_{L^p(w)\to L^p(w)} \\
  &\leq C_{d,p}\|\Omega\|_{L^\infty}\{w\}_{A_p}\sum_{j=0}^\infty (1+N(j))2^{-\alpha_{p,d}N(j-1)/(w)_{A_p}},
\end{align*}
and all that remains is to make a good choice of the increasing function $N(j)$. We choose $N(j)=2^j$ for $j\geq 1$. Then, using $e^x\geq \frac12 x^2$ and hence $e^{-x}\leq 2 x^{-2}$, we have
\begin{equation*}
\begin{split}
  \sum_{j=0}^\infty & (1+N(j))2^{-\alpha_{p,d}N(j-1)/(w)_{A_p}} \\
 & \leq c\sum_{j:2^j\leq(w)_{A_p}}2^j+C_{p.d}\sum_{j:2^j\geq(w)_{A_p}}2^j\Big(\frac{(w)_{A_p}}{2^j}\Big)^2
 \leq C_{p,d}(w)_{A_p}
\end{split}
\end{equation*}
by summing two geometric series in the last step.

This completes the proof that
\begin{equation*}
  \|T_\Omega\|_{L^p(w)\to L^p(w)}\leq C_{d,p}\|\Omega\|_{L^\infty}\{w\}_{A_p}(w)_{A_p}.
\end{equation*}

\begin{remark}
The above bound is the best that one can get by any choice of the function $N(\cdot)$, at least without deeper structural changes in the proof. Indeed, given an increasing function $N:\N\to\N$, let $j_0$ be the smallest value such that $N(j_0)>(w)_{A_p}$, and hence $N(j_0-1)\leq(w)_{A_p}$. But then
\begin{equation*}
  (1+N(j_0))2^{-\alpha_{p,d}N(j_0-1)/(w)_{A_p}}
  \geq (1+(w)_{A_p})2^{-\alpha_{p,d}},
\end{equation*}
so that clearly the entire sum over $j\in\N$ is bigger than the right hand side as well.

It might also be interesting to note that the  na\"ive choice $N(j)=j$ above would have produced a weaker bound, with the second power of $(w)_{A_p}$ rather than the first. This is the reason for us studying the operators $\widetilde{T}_j^N$, instead of just $\widetilde{T}_j$.
\end{remark}

\section{Applications and conjectures}

Let us consider the Ahlfors--Beurling, or just Beurling, operator. This operator $B$ can be understood as a Calder\'on--Zygmund operator, defined on $L^2(\C)$, by
$$
Bf(z):=-\frac{1}{\pi}\PV \int_{\C}\frac{f(z)}{(z-\zeta)^2}\,d A(\zeta),
$$
where $dA$ denotes the area Lebesgue measure on $\C$.

Investigation of this operator was the origin of the $A_2$ conjecture: K. Astala, T. Iwaniec, and E. Saksman \cite{AIS} raised the question whether the norm $\|B\|_{L^p(w)\to L^p(w)}$ depends linearly on $[w]_{A_p}$, $p\ge2$. In particular, they conjectured
\begin{equation}
\label{eq:Blinear}
\|B\|_{L^2(w)\to L^2(w)}\le C[w]_{A_2},
\end{equation}
and this was positively answered in \cite{Pet-Vol}.

Let us now consider the operator given by the integer powers, or composition, of Beurling operators $B^m=B\circ\cdots \circ B$, for $m\in \N$. They have a representation (see e.g. \cite[Eq. (1)]{DPV}, which gives the $m$th power of $\sqrt{B}$ and should be used with $2m$ in place of $m$ for the present purposes) as homogeneous singular integrals as
$$
B^mf(z):=\PV \int_{\C}K_m(z-\zeta) f(\zeta)\,d A(\zeta),
$$
where, for $\zeta=r e^{i\phi}\in \C$,
$$
K_m(\zeta)=\frac{\Omega_m(e^{i\phi})}{|\zeta|^2}, \qquad \Omega_m(e^{i\phi})=\frac{(-1)^m}{\pi}\cdot m\cdot e^{-i2m\phi}.
$$
It is easy to check that $\displaystyle K_1(\zeta)=-\frac{1}{\pi}\frac{1}{\zeta^2}$ and $\|\Omega_m(e^{i\phi})\|_{L^{\infty}}\le m$.

Since each $B^m$ is a nice Calder\'on--Zygmund operator, both the bound \eqref{eq:Blinear} and an analogous bound with $B$ and $C$ replaced by $B^m$ and some $C_m$ are special cases of the general $A_2$ theorem \cite{Hytonen:A2}. Shortly before the general result of \cite{Hytonen:A2}, the operators $B^m$ were studied by O.~Dragi\v{c}evi\'c \cite{Dragi:Bm}, who found that $C_m\leq C\cdot m^3$. By a careful study of the constants in \cite{Hytonen:A2} and subsequent new proofs of the $A_2$ theorem, this could be somewhat improved. From the results in the present paper, we obtain:

\begin{corollary}\label{cor:Bm}
For every $w\in A_p$, we have
$$
\|B^m\|_{L^p(w)\to L^p(w)}\le  C\,m\cdot \{w\}_{A_p}\cdot\min\big(1+\log m,(w)_{A_p}\big),
$$
and in particular
$$
\|B^m\|_{L^2(w)\to L^2(w)}\le  C\,m\cdot [w]_{A_2}\cdot\min\big(1+\log m,[w]_{A_2}\big),
$$
\end{corollary}

\begin{proof}
Observe that, on one hand,
$$
\pi|\Omega_m(e^{i\phi})-\Omega_m(e^{i\phi'})|=m|e^{-i2m\phi}-e^{-i2m\phi'}|=m|e^{-i2m(\phi-\phi')}-1|\le 2m^2|\phi-\phi'|.
$$
Also, $|e^{-i2m\phi}-e^{-i2m\phi'}|\le 2$ obviously. So, if we denote by $\omega_{B^m}(t)$ the modulus of continuity associated with $B^{m}$, we have just proved that
$$
\omega_{B^m}(t)\le C \cdot m\cdot \min\{mt,1\},
$$
where $C$ is a positive constant independent of $B^m$.

Moreover,
$$
\int_0^1\omega_{B^m}(t)\frac{dt}{t}\le C\cdot m\bigg(\int_0^{1/m}mt\frac{dt}{t}+\int_{1/m}^1\frac{dt}{t}\bigg)=C\cdot m(1+\log m),
$$
again, with the constant $C$ independent of $B^m$.
So, by Theorem \ref{thm:main1},
\begin{align}
\label{eq:B1}
\notag\|B^mf\|_{L^p(w)\to L^p(w)}&\le c_p\{w\}_{A_p}\big(\|B^m\|_{L^2\to L^2}+m+m(1+\log m)\big)\\
&\le c_p\cdot m\cdot \{w\}_{A_p}\big(1+\log m\big),
\end{align}
since $\|B^m\|_{L^2\to L^2}\le m$.
On the other hand, if we consider $B^m$ as a rough operator, Theorem \ref{thm:main2} gives
\begin{equation}
\label{eq:B2}
\|B^mf\|_{L^p(w)\to L^p(w)}\le  c_p\cdot m \cdot\{w\}_{A_p} (w)_{A_p}.
\end{equation}
The conclusion follows by considering \eqref{eq:B1} and \eqref{eq:B2} together.
\end{proof}

The form of the bounds above seems too arbitrary to be final, which leads us to conjecture that the last factor should not be needed at all. In order not to obscure the main point by unnecessary technicalities, we state the conjectures only for the case $p=2$ and with the classical $A_2$ constant $[w]_{A_2}$:

\begin{conjecture}
For every $w\in A_2$, we have
$$
\|T_{\Omega}f\|_{L^2(w)}\le  c_d\|\Omega\|_{L^\infty}[w]_{A_2}\|f\|_{L^2(w)}.
$$
\end{conjecture}

In particular, for the operator $B^m$:

\begin{conjecture}
For every $w\in A_2$, we have
$$
\|B^mf\|_{L^2(w)}\le  C\,m\cdot [w]_{A_2}\|f\|_{L^2(w)}.
$$
\end{conjecture}

\appendix

\section{Quantitative form of some classical bounds}

For easy reference, we record several results from the classical Calder\'on--Zygmund theory, in a quantitative form appropriate for our purposes. All these results are in principle well known, but not so easily available with precise quantitative statement.

\begin{theorem}[Calder\'on--Zygmund]\label{thm:CZ}
Let $T$ be an $\omega$-Calder\'on--Zygmund operator whose modulus of continuity satisfies the Dini condition \eqref{eq:Dini}. Then,
\begin{equation*}
  \Norm{T}{L^1\to L^{1,\infty}}\le c_d(\Norm{T}{L^2\to L^2}+\|\omega\|_{\Dini}).
\end{equation*}
\end{theorem}

\begin{proof}[Sketch of proof]
This follows from the usual Calder\'on--Zygmund decomposition technique, which uses smoothness of the kernel in the second variable. The only twist to the usual argument is that, when estimating the size of the level set $\{|Tf|>\lambda\}$, one should make the Calder\'on--Zygmund decomposition of $f$ at the level $\alpha\lambda$ (instead of~$\lambda$) and optimise with respect to $\alpha$ in the end.
\end{proof}

\begin{theorem}[Cotlar's inequality]
Let $T$ be an $\omega$-Calder\'on--Zygmund operator whose modulus of continuity satisfies the Dini condition \eqref{eq:Dini}. If $\delta\in(0,1]$, then
$$
  T_{\sharp}f\leq c_{d,\delta}\big(\Norm{T}{L^2\to L^2}+\|\omega\|_{\Dini}\big)Mf+c_{d,\delta} M_\delta(Tf),
$$
where
$$
M_\delta f(x):=\big(M(|f|^{\delta})(x)\big)^{1/\delta}= \sup_{r>0}\bigg(\fint_{B(x,r)}|f|^{\delta}\bigg)^{1/\delta}
$$
\end{theorem}

\begin{proof}[Sketch of proof]
Fix $x\in\R^d$ and $\eps>0$. For $x'\in B(x,\eps/2)$, we have
\begin{align*}
  T_\eps f(x)
  &=T(1_{B(x,\eps)^c}f)(x) \\
  &=[T(1_{B(x,\eps)^c}f)(x)-T(1_{B(x,\eps)^c}f)(x')]+Tf(x')-T(1_{B(x,\eps)}f)(x').
\end{align*}
Using smoothness of the kernel in the first variable and splitting into dyadic annuli, the first term can be dominated (pointwise in $x'$) by $c_d\|\omega\|_{\Dini}Mf(x)$. Then, we take the $L^\delta$ average over $x'\in B(x,\eps/2)$, namely
\begin{align*}
 |T_\eps f(x)|&\le c_{\delta}\bigg[c_d\big(C_K+\|\omega\|_{\Dini}\big)Mf(x)\\
 &\quad +\Big(\frac{1}{|B(x,\eps/2)|}\int_{B(x,\eps/2)}|Tf(x')|^\delta\,dx'\Big)^{1/\delta}\\
 &\quad +\Big(\frac{1}{|B(x,\eps/2)|}\int_{B(x,\eps/2)}|T(1_{B(x,\eps)}f)(x')|^\delta\,dx'\Big)^{1/\delta}\bigg].
\end{align*}
 The second term gives rise to $M_\delta(Tf)$ by definition. Finally, comparing the $L^\delta$ and $L^{1,\infty}$ norms on a bounded set via Kolmogorov's inequality, and using the boundedness of $T$ from $L^1$ to $L^{1,\infty}$, we obtain the estimate or the last term. Indeed
\begin{align*}
\Big(\frac{1}{|B(x,\eps/2)|} &\int_{B(x,\eps/2)}|T(1_{B(x,\eps)}f)(x')|^\delta\,dx'\Big)^{1/\delta}\le c_{\delta} \frac{\|T(1_{B(x,\eps)}f)\|_{L^{1,\infty}}}{|B(x,\eps/2)|}
\\
&\le c_{\delta}\Norm{T}{L^1\to L^{1,\infty}} \frac{\|1_{B(x,\eps)}f\|_{L^{1}}}{|B(x,\eps/2)|}
\le c_{d,\delta} \Norm{T}{L^1\to L^{1,\infty}} Mf(x).
\end{align*}
Finally, we use the bound for $\Norm{T}{L^1\to L^{1,\infty}}$ from Theorem~\ref{thm:CZ}.
\end{proof}

\begin{corollary}
\label{cor:quantTrunc}
Let $T$ be an $\omega$-Calder\'on--Zygmund operator whose modulus of continuity satisfies the Dini condition \eqref{eq:Dini}. Then
\begin{equation*}
  \Norm{T_{\sharp}}{L^1\to L^{1,\infty}}\le c_d(\Norm{T}{L^2\to L^2}+\|\omega\|_{\Dini}).
\end{equation*}
\end{corollary}

\begin{proof}[Sketch of proof]
We fix some $\delta\in(0,1)$, combine the bounds from the previous two theorems, and use the boundedness of $T:L^1\to L^{1,\infty}$, $M:L^1\to L^{1,\infty}$ and $M_\delta:L^{1,\infty}\to L^{1,\infty}$.
\end{proof}

\section*{Acknowledgments.}
We would like to thank Kangwei Li for helpful comments on the presentation.
The second author wishes to thank the Department of Mathematics and Statistics of the University of Helsinki, and specially the Harmonic Analysis Research Group, for the warm hospitality shown during her visit during Winter-Spring of 2015.


\end{document}